\documentclass{amsart}
\usepackage{latexsym,amsfonts,amsmath,amstext,amssymb,amscd,graphicx}

\newtheorem{thm}{Theorem} 
\newtheorem{lmm}{Lemma}
\newtheorem{prp}{Proposition}
\newtheorem{crl}{Corollary}
\newtheorem{dfn}{Definition}

\newtheorem{rmr}{Remark}

\theoremstyle{definition}

\newcommand{\degm}{\deg_{\Gamma_M}}
\newcommand{\inv}{^{-1}}
\newcommand{\di}{\displaystyle}
\newcommand{\Ghat}{\widehat{G}}
\newcommand{\Hhat}{\widehat{H}}
\newcommand{\Zhat}{\widehat{\ZN^2}}
\newcommand{\oa}{\overline{a}}

\newcommand{\prim}{^\prime}

\newcommand{\Lgrad}{L = \bigoplus_{g\in G}L_g}
\newcommand{\Agrad}{A = \bigoplus_{g\in G}A_g}
\newcommand{\grad}{=\bigoplus_{g\in G}}
\newcommand{\ZN}{\mathbb{Z}}

\newcommand{\Wn}{W(m;\underline{n})}
\newcommand{\Mt}{M(2;\underline{n})}
\newcommand{\Wt}{W(2;\underline{n})}

\newcommand{\oA}{\overline{A}}

\newcommand{\Omn}{O(m;\underline{n})}
\newcommand{\Ot}{O(2;\underline{n})}
\newcommand{\oGamma}{{\overline{\Gamma}}}

\newcommand{\unn}{\underline{n}}
\newcommand{\un}{\underline{1}}
\newcommand{\ut}{\underline{t}}
\newcommand{\us}{\underline{s}}
\newcommand{\ta}{\ut^{3a}}

\newcommand{\tai}{\ut^{3a-3\veps_i}}

\renewcommand{\aa}{\alpha(\ut)}
\newcommand{\ai}{\alpha(\underline{t})\inv}

\newcommand{\ddo}{\partial_1}
\newcommand{\ddt}{\partial_2}
\newcommand{\ddi}{\partial_i}

\newcommand{\tdo}{\til\partial_1}
\newcommand{\tdt}{\til\partial_2}
\newcommand{\tdi}{\til\partial_i}

\newcommand{\G}{\widehat{G}}
\newcommand{\tW}{\widetilde{W}(2;\unn)}

\newcommand{\tE}{\widetilde{E}}
\newcommand{\til}{\widetilde}

\newcommand{\xa}{x^{(a)}}
\newcommand{\xb}{x^{(b)}}

\newcommand{\taun}{\tau(\underline{n})}
\newcommand{\veps}{\varepsilon}

\renewcommand{\div}{\operatorname{div}}

\DeclareMathOperator{\Id}{\operatorname{Id}}
\DeclareMathOperator{\Der}{\operatorname{Der}}

\DeclareMathOperator{\Span}{\operatorname{Span}}

\DeclareMathOperator{\Supp}{\operatorname{Supp}}
\DeclareMathOperator{\Aut}{\operatorname{Aut}}

\begin{document}
\title{Gradings by Groups on Melikyan Algebras}
\author{Jason McGraw}
 
\begin{abstract}In this paper we describe all gradings by abelian groups without elements of order five on the Melikyan algebras over algebraically closed fields of characteristic five. 
\end{abstract}
\maketitle
\section{Introduction}\label{sI}

Let $A$ be an algebra over a field $F$, $G$ a group and $\Aut A$, $\Aut G$ the automorphism groups of $A$ and $G$, respectively. The base field $F$ will always be algebraically closed. The field will be of characteristic five when dealing with Melikyan algebras.

\begin{dfn}\label{d1} A \emph{grading $\Gamma$ by a group $G$} on an algebra $A$, also called a {\em $G$-grading}, is a decomposition  $\Gamma:A = \bigoplus_{g\in G} A_g$ where each $A_g$ is a subspace such that $[A_{g'},A_{g''}]\subset A_{g'g''}$
for all $g',\, g''\in G$. For each $g\in G$, we call the subspace $A_g$ the {\em homogeneous space} of degree $g$. The set $\Supp_\Gamma A=\{ g\in G\, |\, A_g \neq 0\}$ is called the {\em support} of
the grading.
\end{dfn}

For a grading by a group $G$ on a \emph{simple} Lie algebra $L$, it is well known that the subgroup generated by the support is abelian \cite[Lemma 2.1]{typea}. If $L$ is finite-dimensional and the support generates $G$ we have that $G$ is finitely generated.

\begin{dfn}\label{d2}
Two gradings $\di\Agrad$ and
$A=\di\bigoplus_{h\in G}A_h\prim$ of an algebra $A$ are called {\em equivalent} if there exist $\Psi\in\Aut A$ and $\theta\in\Aut G$ such that $\Psi(A_g)=A_{\theta(g)}\prim$
for all $g\in G$.  If $\theta$ is the identity, we call the gradings {\em isomorphic}.\end{dfn}

\begin{dfn}\label{d3} Let $A = \bigoplus_{g\in G}A_g$ be a grading by a group $G$ on an algebra $A$
and $\varphi$ a group homomorphism of $G$ onto $H$. The {\em coarsening of the $G$-grading induced by $\varphi$} is the $H$-grading defined by $A = \bigoplus_{h\in H}\overline{A}_{h}$
where $$\overline{A}_{h}=\di\bigoplus_{g\in G,\ \varphi(g)=h}A_{g}.$$\end{dfn}

The task of finding all gradings on simple Lie algebras by finite groups in the case of algebraically closed fields of characteristic zero is almost complete ---  see \cite{Eld} and also \cite{typebcd,gtwo,typea,typeg,typef,typed,lie2}. In the case of positive characteristic $p$, a description of gradings on the classical simple Lie algebras, with certain exceptions, has been obtained in \cite{BK}, \cite{posb}. In the case of simple graded Cartan type Lie algebras, the gradings by $\ZN$ have been described in \cite{slafpc}. It was shown in \cite{jmmm} that all gradings by groups without elements of order $p$ on the graded simple Cartan type Lie algebras, up to isomorphism, fall into the category of what we call \emph{standard} gradings (which are coarsenings of the standard $\ZN^k$-gradings). In \cite{jmmm} the gradings by arbitrary groups on the Witt algebra $W(1;1)$ were described. The gradings on the restricted Witt and special algebras have been announced recently by Bahturin and Kochetov in \cite{BMK}. This paper will deal with the gradings on the Melikyan algebras by arbitrary abelian groups with no elements of order five in the case where the base field $F$ is assumed to be algebraically closed and $p=5$. We use the notation of \cite{slafpc}.

Our main result is the following.

\begin{thm}\label{t1}
Let $L$ be a Melikyan algebra over an algebraically closed field. Suppose $L$ is graded by a group $G$, the support generates $G$ and $G$ has no elements of order 5. Then the grading is isomorphic to a standard $G$-grading.
\end{thm}

The correspondence between the gradings on an algebra by finite abelian groups of order coprime to the characteristic $p$ of the field and finite abelian subgroups of the automorphisms of this algebra is well known. Using the theory of algebraic groups, this extends to infinite abelian groups. Namely, a grading on an algebra $\Lgrad$ by a finitely generated abelian group without elements of order $p$ gives rise to an embedding of the dual group $\widehat{G}$ into $\Aut L$ using the following action:
$$\chi*y=\chi(g)y,\quad\mbox{for all }y\in L_g,\quad g\in G,\quad\chi\in\G.$$
We will denote this embedding by $\eta:\widehat{G}\to\Aut L$, so
\begin{equation}\label{eq1}\eta(\chi)(y)=\chi*y.\end{equation} 
\begin{lmm}\label{lz}
Let $G$, $H$ be groups, $A$ an algebra and $\phi:G\to H$ be a group homomorphism, $\Gamma:A=\bigoplus_{g\in G}A_g$ be a $G$-grading and $\overline{\Gamma}:A=\bigoplus_{h\in H}\overline{A}_h$ be the $H$-grading defined by $\overline{A}_h=\bigoplus_{g\in G,\ h=\phi(g)}A_g$. Then $\eta_{\overline{\Gamma}}(\widehat{H})\subset\eta_\Gamma(\Ghat)$ where the homomorphisms $\eta_\Gamma:\Ghat\to\Aut A$ and $\eta_{\overline{\Gamma}}:\widehat{H}\to\Aut A$ are defined by (\ref{eq1}) with respect to the gradings $\Gamma$ and $\overline{\Gamma}$ respectively.
\end{lmm}
\begin{proof}
Let $\chi\in\Hhat$. For $y\in A_g$ we have $\eta_{\overline{\Gamma}}(\chi)(y)=\chi(\phi(g))y$ since $A_g\subset\overline{A}_{\phi(g)}$. Let $\zeta:G\to F^\times$ be the map defined by $\zeta(g)=\chi(\phi(g))$ for all $g\in G$. Then $$\zeta(g_1g_2)=\chi(\phi(g_1g_2))=\chi(\phi(g_1)\phi(g_2))=\chi(\phi(g_1))\chi(\phi(g_2))=\zeta(g_1)\zeta(g_2)$$
for all $g_1,g_2\in G$. Hence $\zeta\in\Ghat$. Furthermore, for all $y\in A_g$ we have $$\eta_{\overline{\Gamma}}(\chi)(y)=\chi(\phi(g))y=\zeta(g)y=\eta_\Gamma(\zeta)(y).$$
Hence $\eta_{\overline{\Gamma}}(\chi)\in\eta_\Gamma(\Ghat)$.
\end{proof}

\begin{lmm}\label{l1}
 Let $G$, $H$ be abelian groups without elements of order $p$, $A$ an algebra  $\Gamma:A=\bigoplus_{g\in G}A_g$ be a $G$-grading and $\overline{\Gamma}:A=\bigoplus_{h\in H}\overline{A}_h$ be an $H$-grading such that the groups are generated by their support respectively. If $\eta_\oGamma(H)\subset\eta_\Gamma(G)$ then $\oGamma$ is a coarsening of the $G$-grading where $\eta_\oGamma(H)$ and $\eta_\Gamma(G)$ are defined by (\ref{eq1}).\end{lmm}
\begin{proof}
The eigenspaces of $\eta_\Gamma(G)$ and $\eta_\oGamma(H)$ are $A_g$ and $\oA_h$ respectively for all $g\in \Supp_\Gamma A$ and $h\in \Supp_{\oGamma}A$. Since $\eta_\oGamma(H)\subset\eta_\Gamma(G)$ we have that for any $g\in\Supp_\Gamma A$ the eigenspace $A_g$ of $\eta_\Gamma(G)$ is contained in some eigenspace $\oA_h$ of $\eta_\oGamma(H)$ for some $h\in\Supp_{\oGamma}A$ where $h$ depends on $g$. Let $\phi:\Supp_\Gamma A\to \Supp_\oGamma A$ be the map defined by $\phi(g)=h$ for $g\in\Supp_\Gamma A$ where  $h\in \Supp_{\oGamma}A$ and $A_g\subset \oA_h$. The map $\phi$ extends to a homomorphism of $G$ onto $H$ since $A_gA_{g\prim}\subset A_{g\,g\prim}$ and $\oA_h\oA_{h\prim}\subset \oA_{h\,h\prim}$ by the property of gradings and the groups are generated by their supports respectively. Then $\oGamma$ is a coarsening of $\Gamma$.
\end{proof}

If $L$ is finite-dimensional, then $\Aut L$ is an algebraic group, and the image $\eta(\widehat{G})$ belongs to the class of algebraic groups called quasi-tori. Recall that a \emph{quasi-torus} is an algebraic group that is abelian and consists of semisimple elements. Conversely, given a quasi-torus $Q$ in $\Aut L$, we obtain the eigenspace decomposition of $L$ with respect to $Q$, which is a grading by the group of characters of $Q$, $G=\mathfrak{X}(Q)$.

In this paper, $L$ is a Melikyan algebra $\Mt$, where $\unn=(n_1,n_2)$ is a pair of positive integers --- see the definitions in the next section. Unless it is stated otherwise, $m$ is a positive integer and $\unn=(n_1,\dots,n_m)$ is an $m$-tuple of positive integers. We denote by $a$ and $b$ some $m$-tuples of non-negative integers and by $i,j,k,l$ some integers. 

\section{Melikyan Algebras and Their Standard\\
Gradings}

In this section we introduce some basic definitions, closely following \cite[Chapter 2]{slafpc}. We start by defining the commutative algebras $\Omn$ and the Witt algebras $\Wn$ which we will use to define the Melikyan algebras when $m=2$.

\begin{dfn}\label{d4}
Let $\Omn$ be the commutative algebra $$\Omn:=\left\{\di\sum_{0\leq a\leq\taun}\alpha(a)\xa\;|\;\alpha(a)\in F\right\}$$ over a field of characteristic $p$, where $\taun=(p^{n_1}-1,\dots,p^{n_m}-1)$, with multiplication $$\xa\xb=\binom{a+b}{a}x^{(a+b)},$$
where $\di\binom{a+b}{a}=\di\prod_{i=1}^m\binom{a_i+b_i}{a_i}$.

For $1\leq i\leq m$, let $\epsilon_i:=(0,\dots,0,1,0\dots, 0)$, where the 1 is at the $i$-th position, and $x_i:=x^{(\epsilon_i)}$.
\end{dfn}

There are standard derivations on $\Omn$ defined by $\partial_i(\xa)=x^{(a-\veps_i)}$ for $1\leq i\leq m$. 

\begin{dfn}\label{d5} Let $\Wn$ be the Lie algebra $$\Wn:=\left\{\di\sum_{1\leq i\leq m}f_i\partial_i\;|\;f_i\in\Omn\right\}$$ with the commutator defined by $$[f\partial_i,g\partial_j]=f(\partial_i g)\partial_j-g(\partial_jf)\partial_i,\quad f,g\in\Omn.$$
\end{dfn} 

The Lie algebras $\Wn$ are called \emph{Witt algebras}. $\Wn$ is a subalgebra of $\Der\Omn$, the Lie algebra of derivations of $\Omn$.

From now on the base field $F$ is algebraically closed and its characteristic is 5. We set $\tW=\Ot\tdo+\Ot\tdt$.
We define the map $\div:\Wt\to \Ot$ by 
$$\div(f_1\ddo+f_2\ddt):=\ddo(f_1)+\ddt(f_2)$$
for all $f_1,\ f_2\in \Ot$. Also set $$\di\til{f_1\ddo+f_2\ddt}:=f_1\tdo+f_2\tdt$$ for all $f_1,\ f_2\in \Ot$. 

\begin{dfn}
Let $\Mt:=\Ot\oplus \Wt\oplus\tW$ be the algebra whose multiplication is defined by the following equations. For all $D\in \Wt$, $E\in \tW$ and $f,f_1,f_2,g_1,g_2\in \Ot$ we set
$$\begin{array}{l}
\ [D,\tE]:= \til{[D,E]}+2\div(D)\tE,\\
\\
\ [D,f]:=D(f)-2\div(D)f,\\
\\
\ [f,\tE]:=fE\\
\\
\ [f_1,f_2]:=2(f_1\ddo(f_2)-f_2\ddo(f_1))\tdt+2(f_2\ddt(f_1)-f_1\ddt(f_2))\tdo.\\
\\
\ [f_1\tdo+f_2\tdt,g_1\tdo+g_2\tdt]:=f_1g_2-f_2g_1.
\end{array}$$
We call $\Mt$ the {\em Melikyan algebra}.
\end{dfn}

The algebras $\Omn$, $\Wn$, $\Mt$ defined above have well known canonical $\ZN$-gradings. 

\begin{dfn}Let $A=\Omn$, $\Wn$ or $\Mt$. The {\em canonical $\ZN$-grading of $A$},
$$A=\di\bigoplus_{i\in\ZN}=\{y\in A\,|\,\deg_A(y)=i\},$$
is defined by declaring their degrees, $\deg_O$, $\deg_W$ and $\deg_M$, respectively, as follows:
$$\begin{array}{lcl}\deg_O(x^{(a)})&:=&a_1+\cdots+a_m,\\\deg_W(x^{(a)}\partial_i)&:=&a_1+\cdots+a_m-1,\\
\deg_M(x^{(a)}\partial_i)&:=&3\deg_W(x^{(a)}\partial_i),\\
\deg_M(x^{(a)}\tdi)&:=&3\deg_W(x^{(a)}\partial_i)+2,\\ 
\deg_M(x^{(a)})&:=&3\deg_O(x^{(a)})-2,\end{array}$$
for $0\leq a\leq\taun$. The {\em canonical filtration of $A$}, is defined by declaring $A_{(i)}=\di\bigoplus_{j\geq i}A_i$. 
\end{dfn}
Note that $\Wt=\bigoplus_{i\in\ZN}M_{3i}$.

\begin{lmm}\label{l2}
Let $\oGamma_M:\Mt=\di\bigoplus_{(a_,a_2)\in\ZN^2}M_{(a_1,a_2)}$ where 
$$\begin{array}{lcl}
M_{3(a_1,a_2)}&:=&\Span\{x^{(a+\veps_i)}\ddi\,|\,1\leq i\leq 2\}\\\\
M_{(3a_1,3a_2)+\un}&:=&\Span\{x^{(a+\veps_i)}\tdi\,|\,1\leq i\leq 2\}\\\\
M_{(3a_1,3a_2)-\un}&:=&\Span\{\xa\}.\\
\end{array}$$
The decomposition above is $\ZN^2$-grading on $\Mt.\hfill\square$
\end{lmm}

\begin{rmr}\label{r2}
The support of the $\ZN^2$-grading $\oGamma_M$ does not generate $\ZN^2$. The support generates the subgroup $G=\langle(3i+j,j)\,|\,i,j\in\ZN\rangle$ which is isomorphic to $\ZN^2$. Hence we can define a $\ZN^2$-grading for which the support generates $\ZN^2$. Let $\phi_M:\ZN^2\to\ZN^2$ defined by $\phi_M((1,0))=(3,0)$ and $\phi_M((0,1))=(1,1)$. If we set $L_a=M_{\phi_M(a)}$ for $a\in\ZN^2$ then $\Gamma_M:\Mt=\bigoplus_{a\in\ZN^2}L_a$ is a $\ZN^2$-grading since $\phi_M(\ZN^2)=G$. Also since $L_{(-1,0)}=M_{(-3,0)}=\Span\{\ddo\}$ and $L_{(0,-1)}=M_{(-1,-1)}=F$ we have that the support of the $\Gamma_M$ grading generates $\ZN^2$.
\end{rmr}
Note that the grading in Lemma \ref{l2} is a coarsening of the $\Gamma_M$ grading. By Lemma \ref{lz} we have that $\eta_{\oGamma_M}(\Zhat)\subset\eta_{\Gamma_M}(\Zhat)$. We will mainly work with the grading $\oGamma_M$ and get results for $\Gamma_M$. We will show that $\eta_{\oGamma_M}(\Zhat)$ is a maximal abelian subgroup of $\Aut\Mt$ which implies that $\eta_{\oGamma_M}(\Zhat)=\eta_{\Gamma_M}(\Zhat)$. 
\begin{dfn}
We call the $\ZN^2$-grading $\Gamma_M$ in Remark \ref{r2} the {\em standard $\ZN^2$-grading on $\Mt$}. Let $\degm(y)$ and $\deg(y)$ be the degrees of $y$ with respect to the $\ZN^2$-gradings $\Gamma_M$ and $\oGamma_M$ respectively.
\end{dfn}

\begin{rmr}\label{r1} 
The canonical $\ZN$-grading is a coarsening of the $\ZN^2$-grading $\oGamma_M$ from Lemma \ref{l2} and hence a coarsening of the standard $\ZN^2$-grading $\Gamma_M$. Explicitly, $$M_i=\di\bigoplus_{a_1+a_2=i}M_{(a_1,a_2)}.$$
\end{rmr}

\begin{dfn}\label{d6}
Let $G$ be an abelian group and $\varphi:\ZN^2\to G$ a homomorphism. The decomposition $\Mt=\bigoplus_{g\in G} M_g$, 
given by
$$M_g=\Span\{y\in \Mt\,|\,\varphi(\degm(y))=g\},$$
is a $G$-grading on $\Mt$. We call such decomposition a \emph{standard $G$-grading induced by $\varphi$} on $\Mt$. We will refer to a standard $G$-grading induced by $\varphi$ as a \emph{standard $G$-grading} when $\varphi$ is not specified.  
\end{dfn}

The grading $\oGamma_M$ on $\Mt$ gives rise to a quasi-torus $\eta_{\oGamma}(\Zhat)$. We will show later that $\eta_{\oGamma_M}(\Zhat)$ is actually a maximal torus. Let $\ut^{a}:=t_1^{a_1}t_2^{a_2}$ for all $\ut=(t_1,t_2)\in (F^\times)^2$ and $\aa:=t_1t_2$. We define $\lambda:(F^\times)^2\to\Aut\Mt$ where
$$\begin{array}{lcl}
\lambda(\ut)\xa\ddi&:=&\tai\xa\ddi\\
\lambda(\ut)\xa\tdi&:=&\tai\aa\xa\tdi\\
\lambda(\ut)\xa&:=&\ta\ai\xa.\\
\end{array}$$
For any element $y$ in $M_{(a_1,a_2)}$ of the grading $\oGamma_M$ we have $\lambda(\ut)(y)=\ut^a y$ which is the same as saying $\lambda(\ut)(y)=\ut^{\deg(y)}y$. 

\begin{lmm}\label{l3}$\lambda$ is a homomorphism of algebraic groups.
\end{lmm}
\begin{proof}
We start by showing that for $\ut\in(F^\times)^2$ we have $\lambda(\ut)\in\Aut \Mt$. Lemma \ref{l2} gives us that $\deg([y,z])=\deg(y)+\deg(z)$ when $y$, $z$ are homogeneous elements. For homogeneous $y$, $z$ we have 
$$\begin{array}{lcl}\lambda(\ut)([y,z])&=&\ut^{\deg([y,z])}[y,z]=\ut^{\deg(y)+\deg(z)}[y,z]=\ut^{\deg(y)}\ut^{\deg(z)}[y,z]\\
&=&[\lambda(\ut)(y),\lambda(\ut)(z)].\end{array}$$
Hence $\lambda(\ut)\in\Aut \Mt$.

Now we show that $\lambda$ is a homomorphism. Let $\us$, $\ut\in(F^\times)^2$ and $y$ be a homogeneous element. Then
$$\lambda(\us\,\ut)y=(\us\,\ut)^{\deg(y)}y=\us^{\deg(y)}\ut^{\deg(y)}y=\us^{\deg(y)}\lambda(\ut)(y)=\lambda(\us)(\lambda(\ut)(y))$$
which shows that $\lambda$ is a homomorphism.

It is obvious that $\lambda$ is a rational map and it is a homomorphism.
\end{proof}

Let $T_M:=\lambda((F^\times)^2)$. The kernel of $\lambda$ is $\{(t_1,t_2)\in(F^\times)^2\,|\,t_1^3=t_2^3=1,\,t_1t_2=1\}$. Since the kernel is finite and $\lambda$ is a regular homomorphism we have that $T_M$ is a torus.

\begin{lmm}\label{l4}
The torus $T_M$ is $\eta_{\oGamma_M}(\Zhat)$.
\end{lmm}

\begin{proof}
First we show that $\eta_{\oGamma_M}(\widehat{
\ZN^2})\subset T_M$. Let $\chi\in\Zhat$ and $\chi((1,0))=t_1\in F^\times$ and $\chi((0,1))=t_2\in F^\times$. For $y\in M_{(a_1,a_2)}$ we have $$\begin{array}{lcl}\eta_{\oGamma_M}(\chi)(y)&=&\chi((a_1,a_2))y=\chi((a_1,0))\chi((0,a_2))y\\
&=&\chi((1,0))^{a_1}\chi((0,1))^{a_2}y
=(t_1,t_2)^{\deg(y)}y=\lambda((t_1,t_2))(y).\end{array}$$
Hence $\eta_{\oGamma_M}(\chi)\in T_M$ and we have $\eta_{\oGamma_M}(\Zhat)\subset T_M$. 

Now we show that $T_M\subset\eta_{\oGamma_M}(\Zhat)$. For $\ut=(t_1,t_2)\in(F^\times)^2$ let $\chi_{\ut}:\ZN^2\to F^\times$ be the element of $\Zhat$ defined by $\chi_{\ut}(a)=\ut^a$ for all $a\in\ZN^2$. For $y\in M_a$, $a\in\ZN^2$ we have 
$$\lambda(\ut)(y)=\ut^ay=\chi_{\ut}(a)y=\eta_{\oGamma_M}(\chi_{\ut})(y).$$
Hence $\lambda(\ut)\in\eta_{\oGamma_M}(\Zhat)$ and we have $T_M\subset\eta_{\oGamma_M}(\Zhat)$.
\end{proof}

The following proposition shows that if we want to know more about the quasi-torus $\eta(\G)$ up to conjugation by an automorphism of $\Mt$ then we should look at the normalizer of a maximal torus in $\Aut \Mt$. This follows from \cite[Corollary 3.28]{plat}.

\begin{prp}\label{p1}
A quasi-torus of an algebraic group belongs to the normalizer of a maximal torus. $\hfill\square $
\end{prp}

In Section \ref{sA} we will show that $T_M$ is a maximal torus of $\Aut \Mt$. This leads us to look at the normalizer of the restriction of $T_M$ on $\Wt$ in $\Aut \Wt$. Using that the automorphisms of $\Wt$ can extend to $\Aut\Mt$ (-- see \cite{kuz}) we can then extend the information of the normalizer in $\Aut \Wt$ to get the normalizer of $T_M$ in $\Aut \Mt$.

The goal of Section 3 is to show that if $G$ has no elements of order five then $\eta(\widehat{G})$ is always contained in a maximal torus. 
\section{The automorphism groups of Melikyan algebras}\label{sA}

The automorphism group of $\Mt$ respects the canonical filtration on $\Mt$ (-- see proof of \cite[Theorem 4.7]{kuzz}). Also \cite{kuz} says that any automorphism of $\Wt$ can be extended to an automorphism of $\Mt$. 

 We start by looking at a maximal torus of $\Aut \Wt$. Let 
$$T_W:=\{\psi\in\Aut\Wt\;|\;\psi(\xa\partial_k)=t_1^{a_1} t_2^{a_2}t_k^{-1}\xa\partial_k,\ t_j\in F^\times\}.$$
According to \cite[p. 371]{slafpc}, $T_W$ is indeed a maximal torus of $\Wt$. 

Let $\Aut_W\Mt=\{\Psi\in\Aut\Mt\,|\,\Psi(\Wt)=\Wt\}$ and\\ $\pi:\Aut_W\Mt\to\Aut\Wt$ is the respective restriction map on\\ $\Aut_W\Mt$. Since $T_M=\eta_{\oGamma_M}(\Zhat)$ with respect to the $\ZN^2$-grading $\oGamma_M$ on $\Mt$ and $\Wt$ is a graded subspace of this grading we have\\ $T_M\subset\Aut_W\Mt$.

\begin{lmm}\label{l5}
The restriction of $T_M$ to $\Wt$ is $T_W$. 
\end{lmm}
\begin{proof}
We start by showing $T_W\subset \pi(T_M)$. For any $\psi\in T_W$ we have a pair $(s_1,s_2)\in(F^\times)^2$ such that $\psi(\xa\ddi)=s_1^{a_1}s_2^{a_2}s_i\inv\xa\ddi$. For any element $u$ of $F^\times$ there is at least one element $v$ such that $v^3=u$ because $F$ is algebraically closed. Hence there exist $t_1$ and $t_2$ in $F^\times$ such that $t_1^3=s_1$ and $t_2^3=s_2$. Computing $\lambda(\ut)$ on $\xa\ddi$ we get $$\lambda(\ut)(\xa\ddi)=t_1^{3a_1}t_2^{3a_2}t_i^{-3}\xa\ddi=s_1^{a_1}s_2^{a_2}s_i\inv\xa\ddi.$$
This shows that $\psi=\pi(\lambda((t_1,t_2)))\in \pi(T_M)$ and we have $T_W\subset \pi(T_M)$.

The inclusion $\pi(T_M)\subset T_W$ is obvious.
\end{proof}

The kernel of $\pi$ on $T_M$ is $\{\lambda(\ut)\in T_M\,|\,t_1^3=t_2^3=1\}$.

\begin{lmm}\label{l6}\cite[Lemma 5]{kuz} If $\Theta\in \Aut_W\Mt$ is such that $\pi(\Theta)=\Id_W$ then for $y\in M_i$, $i\in\ZN$, there exists a $\beta$ such that $\Theta(y)=\beta^iy$ where $\beta^3=1.\hfill\square$\end{lmm}

We now fix $\beta$ to be a primitive third root of unity and set $\Theta:=\lambda(\beta^2,\beta^2)$. Note that $\Theta\in T_M$.
\begin{crl}\label{c2}
Let $\Psi$ and $\Phi$ be elements of $\Aut_W\Mt$. If $\pi(\Psi)=\pi(\Phi)$ then there exists an $l$ such that $0\leq l\leq2$ and $\Psi=\Phi\Theta^l$.
\end{crl}
\begin{proof}
If $\pi(\Psi)=\pi(\Phi)$ then $\pi(\Phi\inv\Psi)=\Id_W$. By Lemma \ref{l6} we have $\Phi\inv\Psi=\Theta^l$ for some $0\leq l\leq 2$.
\end{proof}

\begin{crl}\label{c3}
If $\Psi\in \Aut_W\Mt$ is such that $\pi(\Psi)\in T_W$ then $\Psi\in T_M$.
\end{crl}
\begin{proof}
Lemma \ref{l5} shows that there exists $\Phi\in T_M$ such that $\pi(\Psi)=\pi(\Phi)$ and Corollary \ref{c2} says that $\Psi=\Phi\Theta^l$ for some $1\leq l\leq 2$. Hence $\Psi\in T_M$.\end{proof}

In order to describe the normalizers in $\Aut\Wt$ and $\Aut\Mt$ we introduce the automorphism $\upsilon$ of $\Ot$ that induces an automorphism $\sigma$ of $\Wt$ and finally we extend $\sigma$ to $\Aut\Mt$. For $\unn=(n_1,n_2)$ we define $\oa:=(a_2,a_1)$ for $a=(a_1,a_2)\in\ZN^2$. Let $n_1=n_2$. The linear maps $\upsilon$ and $\sigma$ of $\Ot$ and $\Wt$ respectively, defined by
$\upsilon(\xa):=x^{\oa}$ and $\sigma(D):=\upsilon D\upsilon\inv$ for $\xa\in\Ot$, and all $D\in\Wt$.

\begin{lmm}\label{l7}
For $n_1=n_2$, the maps $\upsilon$ and $\sigma$ are automorphisms of $\Ot$ and $\Wt$ respectively.
\end{lmm}
\begin{proof}
It follows easily from \cite[Theorem 6.3.2]{slafpc} that $\upsilon$ is a continuous automorphisms of $\Ot$ (which are in $\Aut\Ot$ as the name implies). It follows from \cite[Theorem 7.3.2]{slafpc} that conjugating an element $D$ of $\Wt$ by a continuous automorphism $\psi$ of $\Ot$, $D\mapsto\psi\circ D\circ\psi\inv$ is an automorphism of $\Wt$ and hence $\sigma\in\Aut\Wt$.
\end{proof}

\begin{lmm}\cite{jmmm}\label{l8} The normalizer of $T_W$ in $\Aut\Wt$ is $T_W$ if $n_1\neq n_2$ and $T_W\langle\sigma\rangle_2$ if $n_1=n_2$.
\end{lmm}

The following follows from the first paragraph on p.3920 of \cite{kuz}.
\begin{prp}\label{p2}
For every automorphism $\psi$ of $\Wt$ there exists a $\psi_M$ of $\Mt$ which respects $\Wt$ and whose restriction to $\Wt$ is $\psi.\hfill\square$
\end{prp}

By Proposition \ref{p2} there exists a  $\sigma_M\in\Aut\Mt$ which respects $\Wt$ and whose restriction is $\sigma$. We fix this $\sigma_M$.

Now we can prove that $T_M$ is a maximal torus in $\Aut\Mt$. Let $U_M=\langle\sigma_M\rangle$ when $n_1=n_2$ and identity otherwise.

\begin{prp}\label{p3} The normalizer of $T_M$ in $\Aut M$ is $T_MU_M$. \end{prp}
\begin{proof}
In \cite[p. 3921]{kuz} it is stated that we can decompose $\Psi\in\Aut M$ as the product of $\Psi=\Phi\Omega$ where $\Phi,$ $\Omega\in\Aut\Mt$ are such that for all $i\in\ZN$ we have
$$\Phi(y)\in y+M_{(i+1)},\quad\mbox{for }y\in M_i,$$
$$\Omega(M_i)=M_i.$$
Let $\Psi\in N_{\Aut\Mt}(T_M)$. Then we will show that $\Phi=\Id_M$.

Let $y\in M_i$ be a nonzero eigenvector of $T_M$. Since $\Psi\in N_{\Aut\Mt}(T_M)$ we have $\Psi(y)$ is an eigenvector. The eigenspaces of $T_M$ are the $M_{(a_1,a_2)}$, where $(a_1,a_2)\in\ZN^2$. Remark \ref{r1} gives us that  $\Psi(y)\in M_k$ for some $k\in\ZN$. Now we use the decomposition of $\Psi=\Phi\Omega$. We have $\Omega(y)=w\in M_i$
since $\Omega(M_l)=M_l$ for all $l$.  
\begin{equation}\label{eq2}\Psi(y)=\Phi\Omega(y)=\Phi(w)\in w+M_{(i+1)}.\end{equation}
Since $w\neq 0$ the calculations above show that $\Psi(y)\in M_{(i)}$ and $\Psi(y)\notin M_{(i+1)}$.

The intersection of $M_k$ and $M_{(i)}=\bigoplus_{j\geq i}M_j$ is zero if $k<i$. Since $\Psi(y)\in M_{(i)}$ by (\ref{eq2}) and $0\neq\Psi(y)\in M_k$ we have $k\geq i$. Since $\Psi(y)\notin M_{(i+1)}$ by (\ref{eq2}) and $\psi(y)\in M_k$ we have $k\leq i$. Hence $k=i$. We have shown that $\Psi(M_i)=M_i$ for all $i$. Hence $\Phi=\Id_M$. Since $\Wt=\bigoplus_{i\in\ZN}M_{3i}$ we have $\Psi(\Wt)=\Wt$. We conclude that if $\Psi\in N_{\Aut\Mt}(T_M)$ then $\Psi$ preserves the standard $\ZN$-grading of $\Mt$ and that $\pi(\Psi)\in N_{\Aut\Wt}(T_W)$ since $\pi(T_M)=T_W$ (Lemma \ref{l5}). 

According to Lemma \ref{l8}, $N_{\Aut\Wt}(T_W)=T_W$ when $n_1\neq n_2$ and $T_W\langle\sigma
\rangle$ if $n_1=n_2$. By Corollary \ref{c3}, the set of automorphisms of $\Mt$ which when restricted to $\Wt$ are in $T_W$ is $T_M$. For $n_1=n_2$, Corollary \ref{c2} says that if $\Psi\in\Aut_W \Mt$ and $\pi(\Psi)=\rho\sigma$, where $\rho\in T_W$ then there exists a $\Xi\in T_M$ such that $\pi(\Xi)=\rho$ and $\Psi=\Xi\sigma_M\Theta^l$ for $0\leq l\leq2$. The automorphism $\Theta$ is in $T_M$ since $\Theta=\lambda(\beta^2,\beta^2)$. Hence, $N_{\Aut\Mt}(T_M)\subset T_MU_M$. 

Conversely let $\Psi\in\Aut_W\Mt$ such that $\pi(\Psi)\in N_{\Aut\Wt}(T_W)$. Then 
$$\pi(\Psi\lambda(\ut)\Psi\inv)=\pi(\Psi)\pi(\lambda(\ut))\pi(\Psi)\inv\in T_W.$$ 
By Corollary \ref{c3}, we have that $\Psi\lambda(\ut)\Psi\inv\in T_M$ and hence it follows that $\Psi\in N_{\Aut\Mt}(T_M)$ . Since for $n_1=n_2$ we have $\pi(\sigma_m)=\sigma\in N_{\Aut\Wt}(T_W)$ and it follows that $\sigma_M\in N_{\Aut\Mt}(T_M)$. We have shown that $$T_MU_M\subset N_{\Aut\Mt}(T_M).$$
\end{proof}

\begin{crl}\label{c4} The centralizer of $T_M$ in $\Aut \Mt$ is $T_M$. Moreover, $T_M$ is a maximal torus and $T_M=\eta_{\Gamma_M}(\Zhat)$.\end{crl}
\begin{proof}
The centralizer of $T_M$ is contained in the normalizer of $T_M$. By Proposition \ref{p3}, the normalizer of $T_M$ is in $\Aut_W\Mt$. This implies that if $\Psi\in\Aut_W\Mt$ and $\Psi$ is in the centralizer of $T_M$ then $\pi(\Psi)$ must be in the centralizer of $\pi(T_M)=T_W$  (Lemma \ref{l5}). Since $T_W$ is a maximal torus and (for $n_1=n_2$) $\sigma\notin T_W$ we have that $\sigma$ is not in the centralizer of $T_W$. Hence $\sigma_M$ is not in the centralizer of $T_M$. We have shown that the centralizer of $T_M$ in $\Aut\Mt$ is $T_M$ and that $T_M$ is a maximal torus. 

By Lemma \ref{lz} we have $T_M=\eta_{\oGamma_M}(\Zhat)\subset\eta_{\Gamma_M}(\Zhat)$. Since $\eta_{\Gamma_M}(\Zhat)$ is abelian and contains $T_M$ it must be in the centralizer of $T_M$ which is $T_M$.\end{proof}

\begin{prp}\label{p4}
Let $Q$ be a quasi-torus in $\Aut\Mt$. There is an automorphism $\Psi\in\Aut\Mt$ such that $\Psi Q\Psi\inv\subset T_M$.
\end{prp}

\begin{proof}
By Proposition \ref{p1}, $Q$ is inside the normalizer of a maximal torus. Up to conjugation we can assume $Q\subset N_{\Aut\Mt}(T_M)$. Then $Q$ must preserve $\Wt$ since $N_{\Aut\Mt}(T_M)=T_MU_M$ (Proposition \ref{p3}). Let $Q\prim=\pi(Q)$. It follows that $Q\prim\subset N_{\Aut\Wt}(T_W)$. In \cite{jmmm} it is shown that there exists a $\psi\in\Aut\Wt$ such that $\psi Q\prim\psi\inv\subset T_W$. Proposition \ref{p2} says that there is $\Psi\in\Aut\Mt$ such that $\pi(\Psi)=\psi$. Hence $\pi(\Psi Q\Psi\inv)=\psi Q\prim\psi\inv\subset T_W$. Since $\pi(\Psi Q\Psi\inv)\subset T_W$, Corollary \ref{c3} give us that $\Psi Q\Psi\inv\subset T_M$. 
\end{proof}
We can now prove Theorem \ref{t1}. 
\begin{proof}
Let $L=\Mt$. Suppose $\Gamma:\Lgrad$ is a $G$-grading where $G$ is a group without elements of order five. Without loss of generality, we assume that the support of the grading generates $G$. Let $\eta_\Gamma:\widehat G\to\Aut L$ be the corresponding embedding and $Q=\eta_\Gamma(\widehat G)$. Then by Proposition \ref{p4}, there is a $\Psi\in\Aut\Mt$ such that $\Psi Q\Psi\inv\subset T_M$. Recall that by Corollary \ref{c4} that $T_M=\eta_{\Gamma_M}(\Zhat)$.
It follows from Lemma \ref{l1} that $L\grad L\prim_g$, where $L\prim_g=\Psi(L_g)$, is a coarsening of the standard $\ZN^2$-grading $\Gamma_M$ which is isomorphic to the original grading.
\end{proof}

\end{document}